\documentclass[11pt]{article}
\usepackage{amsfonts,amsmath,amssymb,amsthm,latexsym,color,enumerate,epsfig,xcolor,mathrsfs,}

\newtheorem{thm}{Theorem}[section]
\newtheorem{dfn}[thm]{Definition}
\newtheorem{lem}[thm]{Lemma}
\newtheorem{cor}[thm]{Corollary}
\newtheorem{prop}[thm]{Proposition}

\newcommand{\mc}[1]{\mathcal{#1}}
\newcommand{\mb}[1]{\mathbb{#1}}

\newcommand{\bfl}[1]{\left\lfloor #1 \right\rfloor}
\newcommand{\bcl}[1]{\left\lceil #1 \right\rceil}

\newcommand{\sub}{\subset}

\newcommand{\sm}{\setminus}
\newcommand{\ov}{\overline}

\newcommand{\eps}{\varepsilon}
\newcommand{\es}{\emptyset}

\newcommand{\aA}{\alpha}
\newcommand{\bB}{\beta}
\newcommand{\gG}{\gamma}
\newcommand{\dD}{\delta}

\newcommand\total{\operatorname{sum}}

\title{Cycle-complete Ramsey numbers}
\author{Peter Keevash\thanks{Mathematical Institute, University of Oxford, Oxford, UK. 
E-mail: keevash@maths.ox.ac.uk. Research supported in part by ERC Consolidator Grant 647678.}
\and Eoin Long\thanks{Mathematical Institute, University of Oxford, Oxford, UK. 
E-mail: long@maths.ox.ac.uk.}
\and Jozef Skokan\thanks{Department of Mathematics, London School of Economics, Houghton Street, London WC2A\thinspace2AE, UK, \emph{and} Department of Mathematics, University of Illinois, 1409 W.\/ Green Street, Urbana IL 61801, USA.  
E-mail: j.skokan@lse.ac.uk. Research supported in part by NSF Grant DMS-1500121.}}
\date{\today}

\newcommand{\np}{\vspace{2mm}}
\newcommand{\noi}{\noindent}
\newcommand{\pow}{\gamma }

\begin{document}

\maketitle

\abstract{
The Ramsey number $r(C_{\ell},K_n)$ is the smallest natural number $N$ 
such that every red/blue edge-colouring of a clique of order $N$ contains
a red cycle of length $\ell$ or a blue clique of order $n$. 
In 1978, Erd\H os, Faudree, Rousseau and Schelp conjectured that 
$r(C_{\ell},K_n) = (\ell-1)(n-1)+1$ for $\ell \geq n\geq 3$ 
provided $(\ell,n) \neq (3,3)$. 

We prove that, for some absolute constant $C\ge 1$, we have 
$r(C_{\ell},K_n) = (\ell-1)(n-1)+1$ provided 
$\ell \geq C\frac {\log n}{\log \log n}$. 
Up to the value of $C$ this is tight since we also show that, for any $\eps >0$ and  $n> n_0(\eps )$, we have
$r(C_{\ell }, K_n) \gg (\ell -1)(n-1)+1$ for all 
$3 \leq \ell  \leq  (1-\eps )\frac {\log n}{\log \log n}$.

This proves the conjecture of Erd\H os, Faudree, Rousseau and Schelp for large $\ell $, 
a stronger form of the conjecture due to Nikiforov, 
and answers (up to multiplicative constants) two further questions of Erd\H os, Faudree, Rousseau and Schelp.
}
\thispagestyle{empty}

\section{Introduction}

Graph Ramsey numbers are a central topic of research in Combinatorics.
Given two graphs $G$ and $H$, the Ramsey number $r(G,H)$ 
is the smallest natural number $N$ such that every red/blue colouring 
of the edges of the complete graph $K_N$ on $N$ vertices 
contains a~red copy of~$G$ or a~blue copy of~$H$. 
The existence of~$r(G,H)$ follows from Ramsey's theorem \cite{Ramsey}, 
but determining or accurately estimating these parameters 
presents many challenging problems.\np 

The classical Ramsey numbers are the graph Ramsey numbers $r(G,H)$
where $G$ and $H$ are cliques. Erd\H{o}s and Szekeres \cite{ES} showed $r(K_n,K_n) \leq 2^{(1+o(1))2n}$,
and later Erd\H{o}s \cite{Erdos-Ram-LB} showed $r(K_n,K_n) \geq 2^{(1+o(1))n/2}$, 
in one of the first instances of the probabilistic method. 
Both bounds changed very little over the past 70 years, 
despite progress by Thomason \cite{Thom} and Conlon \cite{Con} on 
the upper bound, and by Spencer \cite{Spen} on the lower bound. 
Another intensively studied Ramsey number is $r(K_3,K_n)$;
it was a long-standing open problem to determine its order of magnitude,
which is now known to be $\Theta \big ( \frac {n^2}{\log n} \big )$,
due to theorems of Ajtai, Koml\'os and Szemer\'edi \cite{A-K-S} 
and Kim \cite{Kim}. Recent analyses of the triangle-free process 
independently by Bohman and Keevash \cite{B-K} and 
by Fiz Pontiveros, Griffiths and Morris \cite{FGM}, 
together with an improved upper bound due to Shearer \cite{Shearer}, 
have now determined $r(K_3,K_n)$ to within a multiplicative factor of $4+o(1)$.\np 

At the other end of the spectrum, sparse graphs tend to have small Ramsey 
numbers. In this context, Chv\'atal, R\"odl, Szemer\'edi and Trotter \cite{CRST} 
proved that if $G$ and $H$ have bounded maximum degree then 
$r(G,H) = O(v(G)+v(H))$, where $v(G)$ denotes the number of vertices of the 
graph~$G$. A similar bound was obtained by Chen and Schelp  \cite{CS} 
under the assumption of bounded arrangeability. After intense effort 
\cite{A, FS-density, FS-BE, K-R, K-R-II, KS-sparse},
a longstanding conjecture of Burr and Erd\H{o}s \cite{BE} that such bounds hold 
only assuming bounded degeneracy was recently confirmed by Lee \cite{Lee}.\np 

In this paper, we will focus on the cycle-complete Ramsey numbers $r(C_{\ell },K_n)$. 
For any connected graph $H$, Chv\'atal and Harary~\cite{ChvHar} observed that $r(H,K_n) \ge (v(H)-1)(n-1)+1$. This is shown by the red/blue edge-coloured clique of order $(v(H)-1)(n-1)$,  in which the red edges consist of $n-1$ disjoint cliques of order $v(H)-1$ and all the remaining edges are blue.
Burr and Erd\H os \cite{BE-good} asked when equality holds in the Chv\'atal--Harary bound (the `Ramsey goodness' question, see e.g.\ \cite{ABS}).
When $H=C_\ell$, for $\ell \geq n^2-2$
Bondy and Erd\H{o}s \cite{Bondy-Erdos} showed the equality
\begin{equation} \label{eqn: r(C_l ,K_n) conjectured}
r(C_{\ell },K_n) = (\ell -1)(n-1)+1.
\end{equation}
Erd\H{o}s, Faudree, Rousseau and Schelp \cite{EFRS}
noted that $r(C_3,K_n) = r(K_3,K_n)$ grows much faster
than a linear function of $n$ (as discussed above),
and posed the problem of determining the critical $\ell$ 
at which the change in behaviour of $r(C_{\ell },K_n)$ occurs.
They conjectured (see also \cite[Chapter 2]{CG-EoG}) that 
\eqref{eqn: r(C_l ,K_n) conjectured} holds for 
$\ell  \geq n \geq 3$ provided $(\ell ,n) \neq (3,3)$.\np

There is a large literature on $r(C_{\ell },K_n)$.
An improved lower bound on $r(C_{\ell },K_n)$ for small $\ell$ was given by 
Spencer \cite{Spen-CCRN}. Caro, Li, Rousseau and Zhang \cite{C-L-R-Z} improved the upper bound on $r(C_{\ell },K_n)$ of Erd\H{o}s et al. \cite{EFRS} for small even $\ell $;  Sudakov \cite{Sudakov} gave a similar improvement for small odd $\ell $. 
Several authors \cite{FS, rosta, Y-H-Z, BJMRRY, Schiermeyer} confirmed 
the Erd\H{o}s--Faudree--Rousseau--Schelp conjecture for small values of $n$. 
Schiermeyer \cite{Sch} improved the result of Bondy and Erd\H{o}s
by showing that \eqref{eqn: r(C_l ,K_n) conjectured} 
holds for $\ell  \geq n^2 - 2n>3$. 
Nikiforov \cite{Nikiforov} substantially extended this range, 
proving that \eqref{eqn: r(C_l ,K_n) conjectured} holds
for $\ell  \geq 4n +2$.
Moreover, he conjectured (Conjecture 2.14 in \cite{Nikiforov}) that in fact \eqref{eqn: r(C_l ,K_n) conjectured} 
already holds at a much lower threshold, namely that
for all $\eps > 0$ there is $n_0$ such that 
$r(C_{\ell },K_n) = (\ell -1)(n-1)+1$ provided
$\ell \geq n^{\eps }$ and $n \geq n_0$.\np 

Our main result proves both the Erd\H{o}s--Faudree--Rousseau--Schelp conjecture 
for large $\ell $ and Nikiforov's conjecture. 
In fact, we prove \eqref{eqn: r(C_l ,K_n) conjectured} 
for a much wider range of parameters. 

\begin{thm} \label{thm: clique-vs-cycle}
There is $C \ge 1$ so that $r(C_{\ell }, K_n) = (\ell -1)(n-1)+1$ for $n\geq 3$ and $\ell \geq C\frac {\log n} {\log \log n}$.
\end{thm}

\noi {\emph {Remarks:}} 
All logarithms in this paper are to base $2$.
Note that $r(C_{\ell },K_1) = 1$ and 
$r(C_{\ell },K_2) = \ell $ for all $\ell \geq 3$;
we include the condition $n\geq 3$ 
only to avoid division by $0$ 
in the lower bound on $\ell$.\np

The bound in Theorem \ref{thm: clique-vs-cycle} is best possible 
up to the value of $C$, as shown by our next result.

\begin{thm} \label{thm: lower bound}  
Given $\eps >0$ there is $n_0(\eps )$ so that $r(C_{\ell }, K_n) > n\log n \gg (\ell -1)(n-1)+1$ for all $n \geq n_0(\eps )$ and $3 \leq \ell  \leq  (1-\eps )\frac {\log n}{\log \log n}$.
\end{thm}

In combination, Theorems \ref{thm: clique-vs-cycle}
and \ref{thm: lower bound} answer (up to the constant $C$)
two further questions of Erd\H{o}s et al. \cite{EFRS}
regarding $r(C_{\ell },K_n)$, namely 
(i) the location of the critical value of $\ell$ 
for the transition in behaviour of $r(C_{\ell },K_n)$,
and (ii) the choice of $\ell$ that minimises $r(C_{\ell },K_n)$.
The answer to both questions is 
$\ell = \Theta \big ( \frac {\log n}{\log \log n} \big )$.\np

An overview of the proof of Theorem \ref{thm: clique-vs-cycle} and the 
organisation of the paper is as follows. We suppose for a~contradiction 
that there is some $C_\ell$-free graph $G$ with $v(G)=N=(\ell-1)(n-1)+1$
 and independence number $\aA(G) \le n-1$. By induction we can also 
assume $G$ has minimum degree $\dD(G) \ge \ell-1$. The main task of the 
paper is to prove the stability result (Lemma 
\ref{lem: almost clique decomp}) that $G$ is close in structure to the 
lower bound construction described above, i.e.\ $G$ can be mostly 
partitioned into approximate cliques of size about $\ell$ (and also 
less than $\ell$, as there is no $C_\ell$). Then in Section 
\ref{sec:pf}, following various arguments to clean up the approximate 
structure, we will see that it is incompatible with our assumptions,
and so obtain a contradiction that proves the theorem. \np

In the next section, after the short proof of 
Theorem \ref{thm: lower bound}, we gather various tools
needed for the proof of the stability result.
Over the following three sections we prove the
existence of approximate decompositions of $G$
into pieces whose properties are gradually strengthened:
in Section~\ref{sec:dense} the pieces are quite dense,
in Section~\ref{sec:hub} they are `hubs'
(highly connected in a certain sense),
and in Section~\ref{sec:stab} they are `almost cliques',
as required for the stability result. \np

\section{Preliminaries}

We start in the next subsection with some notation, then we prove
Theorem \ref{thm: lower bound}.
In the third subsection we collect various 
well-known results that we use in our proofs.
The final subsection of this section describes 
two applications of Breadth First Search.

\subsection{Notation}
We summarise some (mostly) standard graph theory notation 
(see e.g.\ \cite{Bol}) used in this paper.
Let $G$ be a finite graph.
We write $v(G):=|V(G)|$ for the number of vertices
and $e(G):=|E(G)|$ for the number of edges.
Given a vertex $v \in V(G)$, the neighbourhood
of $v$ in $G$ is $N_G(v):= \{y \in V(G): xy \in E(G)\}$.
The degree of $v$ is $d_G(v):=|N_G(v)|$.
The minimum degree is $\delta (G) := \min \{ d(v): v\in V(G)\}$,
the maximum degree is $\Delta (G) := \max \{d(v): v\in V(G)\}$,
and the average degree is $d(G) := 2e(G)/v(G)$.
Given $A \sub V(G)$, the induced graph $G[A]$
has vertex set $A$ and edge set $\{e \in E(G): e \sub A\}$.
Given disjoint sets $A,B \subset V(G)$, 
we let $G[A,B]$ denote the bipartite graph
with parts $A$ and $B$ and edge set
$\{e \in E(G): |e \cap A| = |e \cap B| = 1\}$.
A path $P = x_0x_1\ldots x_{\ell }$ of length $\ell $ 
consists of $\ell + 1$ distinct vertices $x_{0},\ldots ,x_{\ell }$,
where $x_ix_{i+1}$ is an edge for $i\in \{0,\ldots , \ell -1\}$. 
We call $x_0$ and $x_k$ the end vertices of $P$ 
and say that $P$ is an $x_0x_k$-path. 
We say $P$ is internally disjoint from a set $X$ 
if $X$ contains none of the interior vertices 
$\{x_1,\ldots ,x_{\ell -1}\}$ of $P$.
A cycle of length $\ell$, or $\ell$-cycle, is a graph obtained
from a path $P = x_0x_1\ldots x_{\ell-1}$ of length $\ell-1$ 
by adding the edge $x_{\ell -1}x_0$.
Edges of cycles will often be listed modulo $\ell $, 
so that $x_{\ell -1}x_{\ell }$ represents $x_{\ell -1}x_0$. 
We say $I \sub V(G)$ is independent if $G[I]$ has no edges.
The independence number $\aA(G)$
is the size of a largest independent set in $G$.
Given natural numbers $m\leq n$ we let 
$[m,n] := \{m, m+1,\ldots ,n\}$. To simplify the presentation, 
we may omit floor and ceiling signs when they are not crucial.

\subsection{The lower bound}
\label{subsection: contruction}

The lower bound construction comes from the
following application of the probabilistic method.

\begin{proof}[Proof of Theorem \ref{thm: lower bound}]
Let $\eps \in (0,1)$, $n > n_0(\eps)$ and $N = 2n \log n$. 
It suffices to prove that there is a graph $G$ on at least $N/2$ 
vertices with $\aA (G) < n$ which does not contain a cycle $C_{\ell }$ 
with $\ell  \leq \ell _0 : =(1-\eps ) \frac {\log n }{\log \log n}$. 
We consider a random graph $G_1 \sim G(N,p)$, 
where $p := \frac{3\log \log n}{n-1}$. 
The expected number of independent sets of order $n$ in $G_1$ is 
\begin{equation*}
\binom {N}{n} (1-p)^{\binom {n}{2}} 
\leq \Big ( \frac {eN}{n} e^{-p(n-1)/2} \Big )^n 
= \big ( 2e (\log n ) e^{-3\log \log n /2} \big )^n \ll 1.
\end{equation*}
On the other hand, the expected number of cycles of length at most $\ell _0$ 
is $\sum _{i\in [3,\ell _0]} (Np)^{i} \leq 2(Np)^{\ell _0}  
\leq  2(7 \log n \log \log n)^{(1-\eps )\log n / \log \log n} \ll N$. 
By Markov's inequality applied to both of these expectations,
with positive probability $G_1$ satisfies $\aA (G_1) < n$ 
and has $\leq N/2$ cycles of length at most~$\ell _0$. 
Fixing a choice of such $G_1$ and deleting a vertex 
from each cycle of length at most $\ell _0$ 
leaves a graph $G$ with the required properties.
\end{proof}

\subsection{Tools}

In this subsection we collect several well-known results.
The first is very simple, but we include a short
proof for the convenience of the reader.

\begin{prop} \label{prop: moving to the k-core}
For any graph $G$,
\begin{enumerate}[(i)]
\item if $G$ has no subgraph 
of minimum degree at least $k$ then 
$e(G) \leq \binom {k}{2} + (v(G)-k)(k-1)$;
\item $G$ contains a subgraph $G_1$ with 
$\dD(G_1) \geq d(G)/2$;
\item $G$ contains a bipartite subgraph 
$G_2$ with $d(G_2) \geq d(G)/2$.
\end{enumerate}
\end{prop}

\begin{proof}
To see (i), note that as any subgraph of $G$ contains 
a vertex with degree at most $k-1$, we may iteratively 
delete such vertices until we obtain a subgraph on $k$ vertices. 
The bound follows by counting edges.
Similarly, for (ii), if there were no such $G_1$
we could reduce $G$ to an empty graph by deleting
vertices of degree less than $d(G)/2$, but then
$e(G) < v(G) d(G)/2$ would be a contradiction.
Lastly, for (iii), note that a random induced 
bipartite subgraph $G_2$ of $G$ has 
$\mb{E} d(G_2) = d(G)/2$.
\end{proof}

Next we state several classical results from extremal graph theory.

\begin{thm}[Tur\'an \cite{Turan}] \label{thm: Turan}
Any graph $G$ satisfies $\aA (G) \geq \tfrac {v(G)}{d(G)+1}$.
\end{thm}

\begin{thm}[Dirac \cite{Dirac}]	\label{thm: dirac}
Any graph $G$ with $d(G) \geq v(G)/2$ contains a Hamilton cycle.
\end{thm}

\begin{thm}[Bondy \cite{Bon}] \label{thm: bondy}
Any graph $G$ with $d(G) \geq v(G)/2$ 
is either a complete bipartite graph 
or is pancyclic, i.e.\ contains cycles
of all lengths in $[3,v(G)]$.
\end{thm}

\begin{thm}[Erd\H{o}s and Gallai \cite{EG}] \label{thm: erdos-gallai}
Any graph $G$ with $d(G) > k-1$ has a path of length $k$.
\end{thm}

We conclude by stating a version
of Dependent Random Choice (see \cite[Lemma 7.2]{FS-DRC}). 
	
\begin{thm}\label{thm: DRC}
Given $\eps >0$ there is $\delta >0$ so that the following 
holds for $N \geq N_0(\eps )$ and any $N$-vertex 
graph $G$ with at least $N^{2-\delta }$ edges. 
There are disjoint sets $U_{1}, U_2 \subset V(G)$ 
such that, for $i = 1,2$, every $a,a' \in U_i$ satisfies 
$|N_G(a, U_{3-i}) \cap N_G(a', U_{3-i})| \geq N^{1-\eps }$.
\end{thm}

\subsection{Breadth First Search}

Here give two applications of Breadth First Search,
namely finding short cycles, and a nice decomposition
of a substantial part of any graph.

We start by describing the well-known construction 
of a breadth first search tree $T$ in a graph $G$ 
rooted at some vertex $x \in V(G)$.
At each step $i \ge 0$, we construct a tree $T_i$
with layers $V_0,\dots,V_i$ which are disjoint
subsets of $V(G)$. Initially, $T_0$ is a tree
with one vertex, namely $V(T_0)=V_0=\{x\}$.
Given $T_{i-1}$ for some $i>0$, 
we let $V_i := N_G(V_{i-1}) \sm V(T_{i-1})$.
If $V_i = \es$ we terminate with $T=T_{i-1}$,
otherwise we obtain $T_i$ from $T_{i-1}$
by adding an arbitrary edge of $G$ from each 
vertex in $V_i$ to some vertex in $V_{i-1}$.
It will be useful to consider the first layer
which does not cause the tree to grow significantly,
in the sense of the following simple proposition. \np

\begin{prop} \label{prop: bfs partial decomposition}
Let $\pow > 1$ and let $G$ be an $N$-vertex graph. 
Let $T$ be a breadth first search tree in $G$
rooted at $x\in V(G)$ with layers $V_0,\ldots , V_r$.
Suppose $m \in {\mathbb N}$ is minimal such that 
$|\cup_{i=0}^{m+1} V_i| \leq \gamma |\cup_{i=0}^m V_i|$.  
Then $m \leq \frac {\log N}{\log \pow } = \log_{\pow }(N)$.
\end{prop}

\begin{proof} 
 	By definition of $m$ we have  
 	$N \geq |\bigcup _{i\in [m]}V_i| \geq 
	\pow |\bigcup _{i\in [m-1]}V_i| \geq \ldots \geq 
	\pow ^{m}|V_0| = \pow ^m$.
\end{proof}

Our first application is to finding
short cycles within an approximate range.

\begin{lem}  \label{lem: approx length cycle}
Let $G$ be an $N$-vertex graph with 
$d(G) \geq d = 	16\pow d_1$, where $\pow >1$ and $d_1 \geq 2$. 
Then $G$ contains an $\ell $-cycle for some 
$\ell \in \big [d_1, d_1 + {2\log _{\pow }(N)} \big ]$.
\end{lem}

\begin{proof} 
By Lemma \ref{prop: moving to the k-core} (ii) and (iii) 
there is a bipartite subgraph $G'$ of $G$ with $\delta (G') \geq d(G)/4$. 
Let $T$ be a breadth first search tree in $G'$
rooted at some $x\in V(G')$ with layers $V_0,\ldots , V_r$.
Let $m\leq {\log _{\pow }(N)}$ be as in 
Proposition \ref{prop: bfs partial decomposition}.
As $G'$ is bipartite, we have $G'[V_i] = \es $ for all $i\in [r]$, so 
\begin{equation*}
\sum _{i\in [0,m]}e(G'[V_i,V_{i+1}]) 
= e \big (G'[\cup_{i\in [0,m+1]} V_i] \big ) 
\geq \frac {\delta (G')}{2} \sum _{i\in [0,m]} |V_i|
\geq \sum _{i\in [0,m]} \frac {d(G)}{16\pow } 
 \big (|V_i| + |V_{i+1}| \big ),
\end{equation*}
using $\sum_{i\in [0,m]}(|V_i| + |V_{i+1}|)
\le (1+\gG) \sum_{i\in [0,m]} |V_i|
\le 2\gG \sum_{i\in [0,m]} |V_i|$.
Thus $d(G'[V_i,V_{i+1}]) \geq d(G)/16\pow \ge d_1$ 
for some $i\in [m]$. By Theorem \ref{thm: erdos-gallai},
there is a path of length $d_1$ in $G[V_i,V_{i+1}]$. 
By possibly removing vertices we can obtain an $xy$-path 
in $G[V_i,V_{i+1}]$ of length between $d_1-2$ and $d_1$ 
with $x,y \in V_i$. Combining this with the unique $xy$-path 
in $T$ of length at most $2m \leq 2\log _{\pow }(N)$ 
gives a cycle of length in
$\big [ d_1, d_1 + 2\log _{\pow }(N) \big ]$, as required.
\end{proof}

Our second application is to construct 
the following partial decomposition of a graph $G$,
consisting of a family of disjoint sets $X_i \sub V(G)$,
which are mutually non-adjacent in $G$, 
with each $X_i$ entirely at a fixed distance 
in some tree $T_i$ from the root $x_i$.

\begin{lem} \label{cor: bfs decomposition}
Let $\pow > 1$ and let $G$ be an $N$-vertex graph. 
Then there are triples $\{(x_i,X_i, T_i)\}_{i\in [t]}$,
where each $T_i$ is a subtree of $G$ rooted at $x_i$
and $X_i \sub V(T_i)$, such that:
\begin{enumerate}[(i)]
\item there is $d_i \in [0, \log _{\pow }(N)]$ such that
for all $x_i' \in X_i$ the unique $x_i x'_i$-path 
in $T_i$ has length $d_i$;
\item $\{X_i\}_{i\in [t]}$ are disjoint 
and satisfy $|\bigcup _{i\in [t]} X_i| \geq N/2\pow $;
\item there are no edges of $G$
between $X_i$ and $X_j$ for distinct $i,j\in [t]$.
\end{enumerate}
\end{lem}

\begin{proof}
We prove the statement by induction on $v(G)$,
noting that it is trivial if $v(G) =1$.\np

Let $T$ be a breadth first search tree in $G$
rooted at some $x\in V(G)$ with layers $V_0,\ldots , V_r$.
Let $m\leq {\log _{\pow }(N)}$ be as in 
Proposition \ref{prop: bfs partial decomposition}.
Let $\{X_i\}_{i\in [s]}$ be $\{V_{2i}\}_{2i \in [m]}$ or 
$\{V_{2i+1}\}_{2i+1 \in [m]}$ according to which set 
$\bigcup _{2i \in [m]} V_{2i}$ or 
$\bigcup _{2i+1 \in [m]} V_{2i+1}$ is larger. 
Setting $X := \bigcup_{i\in [s]}X_i$, 
we note that $|X| + |N_G(X)| = 
|X \cup N_G(X)| \leq |\bigcup _{i\in [m+1]} V_i| 
\leq \pow |\bigcup _{i\in [m]} V_i| \leq 2\pow |X|$. \np 

For each $i\in [s]$, set $x_i = x$, $T_i = T$ 
and $d_i = j$, where $X_i = V_j$,
so that (i) holds by the definition of $V_j$.
As $\{X_i\}_{i\in [s]}$ are non-consecutive layers of 
a breadth first search tree, they are disjoint and there are 
no edges between $X_i$ and $X_j$ for distinct $i,j\in [s]$. \np

Now let $W = V(G) \sm (X \cup N_G(X))$ 
and apply induction on $G[W]$ to obtain 
$\{(x_i,X_i, T_i)\}_{i\in [s+1,t]}$. 
We claim that $\{(x_i,X_i, T_i)\}_{i\in [t]}$ 
satisfy the statement of the lemma.
Indeed, (i) holds by construction.
For (ii), disjointness is clear, 
and we have $\sum _{i\in [s]} |X_i| = 
|X| \geq {|X \cup N_G(X)|}/{2\pow }$
and $\sum _{i\in [s+1,t]} |X_i| \geq {|W|}/{2\pow }$ by induction.
Finally, (iii) holds by construction 
and as each $X_j$ with $j\in [s+1,t]$
is contained in $W$, which is disjoint from $X \cup N_G(X)$.
\end{proof}

\section{Quite dense subgraphs} \label{sec:dense}

In this section we take our first steps towards
the stability result described above,
by showing that any supposed counterexample
to Theorem \ref{thm: clique-vs-cycle} can be
partitioned almost entirely into vertex-disjoint subgraphs, 
each of which is quite large (has $\ell ^{1-o(1)}$ vertices) 
and is quite dense (has $\ell ^{2-o(1)}$ edges). 
We start by showing that a graph of large minimum degree
has a long path or a dense subgraph.

\begin{lem}
\label{lem: long cycle or dense subgraph}
Fix $D \in \mb{N}$. Then any graph $G$ either
\begin{enumerate}[(i)]
\item contains paths of length 
at least $D$ starting at any given vertex, or
\item has a subgraph $H$ with $v(H) \leq D$ 
and $e(H) \geq \binom {\dD(G)+1}{2}$.
\end{enumerate}  
\end{lem}

\begin{proof}
Suppose that (i) fails, i.e.\ there is $x_0 \in V(G)$ 
such that any path starting at $x_0$ has length less than $D$.
We must show that (ii) holds. 
We construct a path $P$ starting at $x_0$ as follows.
At step $i \ge 0$, having chosen a path $P_{i-1}=x_{0} \ldots x_{i-1}$,
we select $x_i \in N_G(x_{i-1}) \sm \{x_0,\ldots ,x_{i-1}\}$ 
that maximises $|N_{G}(x_i) \cap \{x_0,\ldots , x_{i-1}\}|$.
If no such $x_i$ exists we terminate with $P=P_{i-1}$.
Let $P = x_0x_1\cdots x_{\ell }$ be the final path,
where by choice of $x_0$ we have $\ell < D$. 
By the termination rule, we have $N_G(x_{\ell }) \subset V(P)$.
Let $N_G(x_{\ell }) = \{ x_{i_1},\ldots ,x_{i_s} \}$,
where $s \ge \dD(G)$, ordered so that $i_1 < \ldots < i_s$. 
As $x_{\ell }$ is adjacent to $x_{i_j}$ for each $j\in [s]$, 
the rule for choosing $x_{i_j}$ guarantees 
$|N_{G}(x_{{i_j}+1}) \cap \{x_0,\ldots , x_{{i_j}}\}| 
\geq |N_{G}(x_{\ell }) \cap \{x_0,\ldots , x_{{i_j}}\}| = j$ 
for each $j\in [s]$. Then $H = G[V(P)]$ satisfies $v(H) \leq D$ 
and $e(H) \geq \sum _{j=1}^s j \ge \tbinom{\dD(G)+1}{2}$.
\end{proof}

\noi \textit{Remark:} 
An unpublished result of the second author in \cite{L-thesis} 
used a variant of Lemma \ref{lem: long cycle or dense subgraph} 
to prove that subgraphs of the cube graph with average degree $d$ 
contain paths and cycles of length at least~$2^{\Omega (\sqrt d)}$. 
This result was later improved to~$2^{\Omega (d)}$ in \cite{L-cube} 
via a~different approach. \np

We combine the previous lemma with two applications 
of the breadth first search
decomposition of the previous section to show that
any $C_\ell$-free graph with small independence number
contains a~small dense subgraph.

\begin{lem}
\label{lem: indep bound for cycle or dense subgraph}
Let $N,D,\ell \in {\mathbb N}$, $\pow >1$,  
where $3\log _{\pow }(N) \le \ell \le D$, and $d \geq 8 \gamma ^2$. 
Suppose $G$ is a~$C_\ell$-free graph 
on $N$ vertices with $\aA(G) \leq N/d$. 
Then $G$ has a subgraph $H$ with $v(H) \le D$
and $e(H) \ge {d^2}/{2^9 \pow ^4}$.
\end{lem}

\begin{proof}
Let $\{(x_i,X_i,T_i)\}_{i\in [t]}$ be obtained by
applying Lemma \ref{cor: bfs decomposition} to $G$.
Let $X = \bigcup _{i\in [t]} X_i$,
and note that $|X| \geq {N}/{2\pow }$.
Let $\{(y_i,Y_i,T_i')\}_{i\in [s]}$ be obtained by
applying Lemma \ref{cor: bfs decomposition} again,
this time to $G[X]$. 
Let $Y = \bigcup _{i\in [s]} Y_i$, and note that
$|Y| \geq {|X|}/{2 \pow } \geq {N}/{4\pow ^2} 
\geq {d\aA (G)}/4 \pow ^2$. 
By Theorem \ref{thm: Turan} (Tur\'an's Theorem),
$d(G[Y]) \geq {d}/4\pow ^2 - 1 \geq d / 8 \pow ^2$, 
as $d \geq 8\gamma ^2$. 
Then Proposition \ref{prop: moving to the k-core} (ii)
applied to $G[Y]$ gives some $G'=G[Y']$ with $Y' \subset Y$ 
such that $\dD(G') \geq {d}/{16 \pow ^2}$.
By Lemma \ref{lem: long cycle or dense subgraph},
to complete the proof of the lemma, it suffices
to show that $G'$ does not contain a path of length $D$.\np
 
For contradiction, suppose $P = z_0z_1\ldots z_D$ is 
a path in $G'$. As $z_0 \in Y$ there is a triple 
$(y_j, Y_j, T_j')$ with $z_0 \in Y_j$. 
As $T_j'$ is a tree, and so a connected subgraph of $G[X]$, 
by Lemma \ref{cor: bfs decomposition} (iii)
there is a triple $(x_i, X_i, T_i)$ with $V(T_j') \subset X_i$,
and by (i) there is $d_i \in [0, {\log _{\pow }(N)}]$ 
so that every vertex in $X_i$ is at distance $d_i$
from $x_i$ in $T_i$. In particular, the $x_i y_j$-path
and $x_i z_0$-path in $T_i$ only intersect $X_i$
in $y_j$ and $z_0$. We let $P_1$ be the $y_jz_0$-path in $T_i$.
Then $P_1$ has length $\ell_1 \le 2 \log _{\pow }(N)$
and intersects $X_i$ only in $y_j$ and $z_0$.\np

We now use the triple $(y_j, Y_j, T_j')$.
As $P$ is a connected subgraph of $G[Y]$, 
by Lemma \ref{cor: bfs decomposition} (iii) 
we have $V(P) \subset Y_j$, and by (i)
there is $d_j' \in [0, \log _{\pow }(N)]$ so that 
every vertex of $Y_j$ is at distance
$d'_j$ from $y_j$ in $T_j'$.
Let $\ell _2 = \ell - \ell _1 - d_j'$ and consider 
the subpath $P_2 = z_0 z_1 \ldots z_{\ell _2}$ of $P$.
Let $P_3$ be the $y_j z_{\ell _2}$-path in $T_j'$.
Then $P_3$ has length $d'_j$
and intersects $Y_j$ only in $z_{\ell _2}$.
As $V(P_3) \sub V(T_j') \sub X_i$,
we can combine $P_1, P_2$ and $P_3$ 
to form a cycle of length $\ell$.
This contradiction completes the proof.
\end{proof}
	
By iterating the previous lemma one can obtain
the following approximate decomposition of the vertex set of $G$.
This Corollary will not be used in the proof of 
Theorem \ref{thm: clique-vs-cycle} so we omit its proof, which is similar 
to that of Corollary \ref{cor: decomp into hubs} in the next section.

\begin{cor} \label{cor: dense decomp}
Given $\eps > 0$ there is $C\ge 1$ so that the 
following holds for all $\ell , n \in {\mathbb N}$ with 
 $n\geq 3$  and $\ell  \geq C\frac {\log n}{\log \log n}$.
Suppose $G$ is a $C_{\ell }$-free graph on 
$N = (\ell -1)(n-1)+1$ vertices with $\aA (G) \leq n-1$. 
Then there is a partition 
$V(G) = W \cup \bigcup _{i\in [L]} V_i$ 
so that $|V_i| < \ell$ and $e(G[V_i]) > \ell^{2-\eps}$
for all $i\in [L]$, and $|W| \leq \eps N$.
\end{cor}

\section{Hubs} \label{sec:hub}	

Continuing our progress towards the stability result,
we next upgrade the properties of our decomposition
by showing that the quite dense pieces from the last
section must contain quite large `hubs',
which have the property that any small set of vertices 
can be joined together via disjoint paths of 
essentially any desired lengths.
The precise definition is as follows.

\begin{dfn} \label{def:hub}
Let $G$ be a graph and $A,B \subset V(G)$ be disjoint sets. 
Given distinct $x,y \in A \cup B$, we call $\ell \in \mb{N}$ 
a \emph{bipartite length for} $\{x,y\}$ \emph{in} $G[A,B]$ 
if (a) $\ell $ is even and $\{x,y\} \subset A$ or $\{x,y\} \subset B$, 
or (b) $\ell $ is odd and $|\{x,y\} \cap A| = |\{x,y\}\cap B| = 1$.\np	
	
For $\eps\in(0,1)$ and $u \in \mb{N}$, we call a triple $(A,B,D)$ an 
$(u, \eps )$-hub in a graph $G$ if $|A|=|B| = u$, $|D| \leq \eps u$, 
and for any distinct $s_1,\ldots ,s_m, t_1,\ldots , t_m$ in $A \cup B$
with $m \le u^{1-\eps }$ we have the following connection property:
for any $\ell _1, \ldots ,\ell _m \geq 2$ with 
$\sum _{i\in [m]} (\ell _i + 1) \leq 2(1-\eps )u$, where each 
$\ell _i$ is a~bipartite length for $\{s_i,t_i\}$ in $G[A,B]$, 
there are vertex-disjoint paths $P_1,\ldots, P_m$ in $G[A \cup B \cup D]$, 
where each $P_i$ is an $s_it_i$-path of length $\ell _i$. 
\end{dfn}

The main lemma of this section
shows that quite dense graphs contain large hubs.

\begin{lem} \label{lem: locating hubs in G}
Given $\eps \in (0,1)$ there is $\delta >0$ so that for $N \geq N_0(\eps )$ 
and any integer $u \in [N^{\eps }, N^{1-\eps }]$, every $N$-vertex 
graph $G$ with $d(G) \geq N^{1-\delta }$ contains a $(u,\eps )$-hub.	
\end{lem}
	
\begin{proof}
We assume throughout the proof that
$\dD$ is sufficiently small and $N$ is sufficiently large.
By Proposition \ref{prop: moving to the k-core} (iii)
we may assume $G$ is bipartite. Let $\dD_1 = \eps^2/10$.
By Theorem \ref{thm: DRC}, applied with $\dD_1$ in place of $\eps$,
there are disjoint $U_1, U_2 \subset V(G)$ so that
$|N_G(a, U_{3-i}) \cap N_G(a', U_{3-i})| \geq N^{1-\dD_1}$ 
for every $a,a' \in U_i$ with $i \in [2]$. 
As $G$ is bipartite, $U_1$ and $U_2$ must
lie on opposite sides of the bipartition. 
We construct an alternating cycle $C$ of length $2u$ 
in $G[U_1,U_2]$ by fixing distinct vertices $a_1,\ldots ,a_u \in U_1$
and greedily selecting a common neighbour in $U_2$ of each
consecutive pair $\{a_i,a_{i+1}\}$ (including $\{a_u,a_1\}$) 
so that all selected vertices are distinct. 
This is possible as $u \leq N^{1-\eps } \ll N^{1-\delta _1}$.
We let $A = V(C) \cap U_1 = \{a_1,\ldots,a_u\}$
and $B = V(C) \cap U_2$. \np 

We let $D$ be a random subset of $(U_1 \cup U_2) \sm (A \cup B)$ 
where each element is included independently with probability 
$p = \eps u / 2 N$. By Markov's inequality, 
$|D| \leq 2pN \leq \eps u$ with probability at least~$1/2$. 
Furthermore, for each pair $a,a' \in A$, we have 
$$
 {\mathbb E}\big(|N_G(a) \cap N_G(a') \cap D|\big)
 \geq 
 p\big(|N_G(a, U_{2}) \cap N_G(a', U_{2})| - |U_2 \cap C|\big)
 \geq 
\eps u/4N^{\delta _1} \geq 2u^{1-\eps  /2},
$$
and similarly for each pair in $B$.
By Chernoff's inequality (see \cite[Appendix A ]{AS}),
with positive probability $D$ satisfies $|D| \leq \eps u$ and 
$|N_G(c) \cap N_G(c') \cap D| \geq u^{1-\eps /2}$ 
for all $\{c,c'\} \subset A$ or $\{c,c'\} \subset B$.
We fix any set $D$ with these properties.\np
		
It remains to show that $(A,B,D)$ is a $(u, \eps )$-hub.
Suppose $S = \{s_1,\ldots ,s_m\}$ and $T = \{ t_1,\ldots , t_m \}$ 
are disjoint subsets of $A \cup B$ with  $m \le u^{1-\eps }$.
Let $\ell _1, \ldots ,\ell _m \geq 2$ with 
$\sum _{i\in [m]} (\ell _i + 1) \leq 2(1-\eps )u$, where each 
$\ell _i$ is a bipartite length for $\{s_i,t_i\}$ in $G[A,B]$.
We want to find vertex-disjoint paths $P_1,\ldots, P_m$ in $G[A \cup B \cup D]$, 
where each $P_i$ is an $s_it_i$-path of length $\ell _i$. \np

First we claim that there is a path $R$ 
with $V(R) \cap (S \cup T) = \es$,
$|V(R) \cap D| \le 2m$ and $|(A \cup B) \sm V(R)| \le 4m$.
To see this, we consider $C \sm (S \cup T)$, which is 
the vertex-disjoint union of some paths $R_1,\ldots,R_k$,
where $k \le 2m$. By deleting at most two vertices from each 
such path $R_i$, we can assume that each starts and ends in $A$.
We form $R$ by `stitching' these paths together greedily,
using distinct vertices from $D \cap B$ to link successive 
paths $R_i$ and $R_{i+1}$ for all $i \in [k-1]$. 
This is possible by the common neigbourhood property,
as $2m \ll u^{1-\eps /2}$, so the claim follows. \np

Now we will construct the paths $P_1,\ldots, P_m$
by chopping $R$ into suitable subpaths and connecting
these to the endpoint sets $S$ and $T$.
To construct $P_1$, we consider separately the cases
$\ell_1=2$, $\ell_1=3$ and $\ell_1 \ge 4$.
If $\ell_1=2$ we let $P_1 = s_1u_1t_1$ for any 
common neighbour $u_1 \in D$ of $s_1$ and $t_1$ 
disjoint from all previous choices.
If $\ell_1=3$ we let $P_1 = s_1u_1v_1t_1$ 
where $u_1 \in D$ is a neighbour of $s_1$
and $v_1 \in D$ is a common neighbour of $t_1$ and $u_1$,
with $\{u_1,v_1\}$ disjoint from  all previous choices.
Lastly, if $\ell _1 \geq 4$ we consider a subpath $R_1$ 
starting at one end of $R$ with length $\ell_1-3$.
As $\ell _1$ is a bipartite length for $\{s_1,t_1\}$, 
it is possible to delete a vertex from one end of $R_1$ 
to obtain a subpath $R_1'$ of length $\ell _1 - 4$ 
which starts on the same side of the partition as $s_1$ 
and ends on the same side as $t_1$.
Writing $x_1$ and $y_1$ for the ends of $R_1'$,
we form the $s_1t_1$-path $P_1$ of length $\ell _1$
from $R_1'$ by adding paths $s_1u_1x_1$ and $t_1v_1y_1$ where 
$u_1 \in D$ is a common neighbour of $s_1$ and $x_1$,
and $v_1 \in D$ is a common neighbour of $t_1$ and $y_1$,
with $\{u_1,v_1\}$ disjoint from all previous choices.
To continue, we modify $R$ by removing $R_1$,
then repeat the process to find $P_2$, and so on. \np

It remains to show that the above process succeeds,
i.e.\ that we do not ever exhaust $R$ or any
common neighbourhoods in $D$. To see this, 
note that initially $|R| \geq |\bigcup _{i\in [k]} R_i| 
\geq 2u - |S \cup T| - 2m \geq \sum _{i\in [m]} \ell _i$. 
As we remove at most $\ell _i$ vertices from $R$ to build each path $P_i$, 
we never run out of vertices in $R$. Also, we used at most $2$ vertices 
from $D$ to build each $P_i$, and so at most $4m \leq u^{1-\eps /2}/2$ 
from $D$ in total. As $|N_G(a) \cap N_G(a')\cap D| \geq u^{1-\eps /2}$ 
for all $\{a,a'\} \subset A$ or $\{a,a'\} \subset B$, 
we never run out of common neighbours in $D$.
\end{proof}

We conclude this section by showing that any supposed 
counterexample to Theorem \ref{thm: clique-vs-cycle} 
can be partitioned almost entirely into quite large hubs.

\begin{cor}
\label{cor: decomp into hubs}
Given $\eps > 0$ there is $C\ge 1$ so that the following holds 
for all $\ell , n \in {\mathbb N}$ with $n\geq 3$ and 
$\ell  \geq C\frac{\log n}{\log \log n}$.
Suppose $G$ is a $C_{\ell }$-free graph on 
$N = (\ell -1)(n-1)+1$ vertices with $\aA (G) \leq n-1$. 
Then there is a partition 
$V(G) = W \cup \bigcup _{i\in [L]} (A_i \cup B_i \cup D_i)$ 
so that $|W| \leq \eps N$ and each $(A_i,B_i,D_i)$ 
is a $(u,\eps )$-hub with $u := \ell^{1-\eps }$.
\end{cor}

\begin{proof}
Let $\dD>0$ be such that 
Lemma \ref{lem: locating hubs in G} 
applies with $\eps/2$ in place of $\eps$.
Let $\bB = \dD/7$ and $C \ge 4/\bB$ be sufficiently large.
It suffices to show that any $W \subset V(G)$ with 
$|W| > \eps N$ contains a $(u,\eps )$-hub, 
as then iteratively removing such hubs proves the lemma.\np 

To see this, we claim that we can apply 
Lemma \ref{lem: indep bound for cycle or dense subgraph} 
to $G[W]$ with $\pow = \ell^{\bB}$, $D=\ell$ and $d = \ell^{1-\bB}$.
Indeed, for $C$ large we have $d \ge 8\pow^2$
and $\aA (G) \leq n-1 \leq |W|/d$, and also $D = \ell \geq 
\frac {3 \log (N)}{\log ( \ell ^{\bB } )} \geq {3\log _{\pow }(|W|)}$, 
as $\ell \geq (4/\beta )\frac {\log n }{\log \log n}$.
Thus Lemma \ref{lem: indep bound for cycle or dense subgraph} 
gives a subgraph $H$ of $G[W]$ with $v(H) \le \ell$
and $e(G[U]) \geq d^2/{2^9\pow ^4} = 2^{-9}\ell^{2-6\bB}
\geq \ell^{2-\dD}$. Now Lemma \ref{lem: locating hubs in G} 
gives a $(u,\eps )$-hub in $G[W]$.
\end{proof}	

\section{Stability} \label{sec:stab}

In this section we upgrade the decomposition provided by
Corollary \ref{cor: decomp into hubs} to obtain our main
stability result, namely that any supposed 
counterexample to Theorem \ref{thm: clique-vs-cycle} 
can be partitioned almost entirely into quite large
approximate cliques, and furthermore there are no edges
between parts. The precise statement is as follows.

\begin{lem} \label{lem: almost clique decomp}
Given $\eta > 0$ there is $C\ge 1$ so that the 
following holds for all $\ell , n \in {\mathbb N}$ with $n\geq 3$ 
and $\ell  \geq C\frac {\log n}{\log \log n}$.
Suppose $G$ is a $C_{\ell }$-free graph on 
$N = (\ell -1)(n-1)+1$ vertices with $\aA (G) \leq n-1$. 
Then there are disjoint sets 
$V_1,\ldots, V_s \subset V(G)$ such that: 
\begin{enumerate}[(i)]
\item $|V_i| \in [(1-\eta ) \ell , \ell ]$ for all $i\in [s]$;
\item $|\bigcup _{i\in [s]} V_i | \geq (1-\eta )N$;
\item $G[V_i]$ has minimum degree 
at least $(1-\eta )\ell$ for all $i\in [s]$;
\item There are no edges of $G$ between $V_i$ and $V_j$
for all distinct $i,j \in [s]$.
\end{enumerate}
\end{lem}

\noi Throughout the section we will fix $G$ 
as in Lemma \ref{lem: almost clique decomp},
with $\eps<\eps_0(\eta)$ sufficiently small
and $C$ sufficiently large so that
Corollary \ref{cor: decomp into hubs} gives a partition 
$V(G) = W \cup \bigcup _{i\in [L]} (A_i \cup B_i \cup D_i)$ 
with $|W| \leq \eps N$, where each $(A_i,B_i,D_i)$ 
is a $(u,\eps )$-hub with $u = \ell^{1-\eps }$.\np

The proof proceeds in several stages, 
gradually refining the structure provided from the hubs 
to that in Lemma \ref{lem: almost clique decomp}.
In the next subsection we show how to find cycles
of specified lengths in a system of hubs and `handles'
(suitable paths connecting the hubs). There is a potential
parity obstacle due to the bipartite structure of hubs,
but we can eliminate this obstacle using the bound on $\aA (G)$;
this is achieved in the second subsection.
In the third subsection we study the interaction between hubs:
roughly speaking, we consider an auxiliary graph $H_3$,
where $V(H_3)$ consists of most of the hubs and
we join two hubs if they are connected by a large matching.
We show that $H_3$ cannot have large components,
and then in the final subsection we show that these components
identify the approximate cliques needed to prove
Lemma \ref{lem: almost clique decomp}.

\subsection{Cycles from hubs and handles}

In this subsection we show how to find cycles from
a suitable system of hubs and connecting paths.
Our first lemma concerns the following condition under which we can
drop the parity restriction on lengths of paths within a hub.
We say that a $(u,\eps )$-hub $(A,B,D)$ is \emph{parity broken} 
if $G[A]$ contains a matching of size $2u^{1-\eps }$.\np

\begin{lem} \label{lem: parity broken hub lem} 
Suppose $(A,B,D)$ is a parity broken $(u,\eps)$-hub in $G$.
Let $s_1,\ldots ,s_m, t_1,\ldots , t_m \in A \cup B$ be distinct 
and $\ell _1, \ldots ,\ell _m \geq 2$ with 
$\sum _{i\in [m]} (\ell _i + 1) \leq 2(1-\eps )u$.
Suppose also that, for each $i \in [m]$, if $\ell_i$ 
is not a bipartite length for $\{s_i,t_i\}$ in $G[A,B]$ 
then $\ell_i \ge 7$.
Then there are vertex-disjoint paths $P_1,\ldots, P_m$ 
in $G[A \cup B \cup D]$, where each $P_i$ 
is an $s_it_i$-path of length $\ell _i$.
\end{lem}

\begin{proof}
As $(A,B,D)$ is parity broken and $2u^{1-\eps }-2m \geq m$,  
there is a matching ${\cal M} = \{x_iy_i: i\in [m]\}$ in $G[A]$ 
which is vertex-disjoint from $\{s_1,\ldots, s_m, t_1,\ldots ,t_m\}$. 
We will apply the connection property of $(A,B,D)$ to a collection
of pairs $(s_{i,k},t_{i,k})$ where there are one or two pairs
for each original pair $(s_i,t_i)$. 
If $\ell _i$ is a bipartite length for $\{s_i,t_i\}$ 
then we take one pair $(s_{i,1},t_{i,1})=(s_i,t_i)$
with the same length $\ell_{i,1} = \ell_i$.
Otherwise, we take two pairs
$(s_{i,1},t_{i,1})=(s_i,x_i)$ and 
$(s_{i,2},t_{i,2})=(y_i,t_i)$ with lengths
$\ell_{i,1}, \ell_{i,2} \geq 2$ chosen such that 
both $\ell_{i,k}$ are bipartite lengths 
for $\{s_{i,k}, t_{i,k}\}$ in $G[A,B]$ 
with $\ell_{i,1} + \ell_{i,2} + 1 = \ell _i$. 
By the connection property of $(A,B,D)$
we find vertex-disjoint $s_{i,k}t_{i,k}$-paths
of lengths $\ell_{i,k}$, which combine
with edges from ${\cal M}$ to produce 
the required paths $P_1,\ldots, P_m$.	
\end{proof}

Let $\mc{H}$ be a set of vertex-disjoint $(u,\eps)$-hubs
and $\mc{P}=\{P_1,\dots,P_k\}$ be a set
of vertex-disjoint paths in a graph $G$.
Suppose $P_i$ is an $b_ia_{i+1}$-path for $i \in [k]$,
writing $a_{k+1}:=a_1$.
We call $\mc{P}$ a {\em handle system} for $\mc{H}$ if
\begin{enumerate}[(i)]
\item each $P_i$ is internally disjoint 
from $\bigcup \{ V(H): H \in \mc{H} \}$,
\item for each $i \in [k]$ there is $H_i \in \mc{H}$
with $\{a_i,b_i\} \sub V(H_i)$,
\item each $H \in \mc{H}$ contains 
at most $u^{1-\eps}/2$ of $\{a_1,\dots,a_k\}$.
\end{enumerate}
Note that we often apply the above definition with
some paths $P_i$ consisting only of the edge $b_ia_{i+1}$
(in which case condition (i) is vacuous).
The next lemma shows how handle systems
provide cycles of specified lengths.

\begin{lem} \label{lem: hub lemma}
Let $\mc{H}$ be a set of vertex-disjoint $(u,\eps)$-hubs in $G$ 
and $\mc{P}=\{P_1,\dots,P_k\}$ be a handle system for $\mc{H}$,
where each $P_i$ is an $b_ia_{i+1}$-path of length $\ell_i$.
Let ${\ell}_{\total} = \sum _{i\in [k]} \ell _i$. Then:
\begin{enumerate}[(i)]
\item  If $\{a_i, b_i\} \sub A_i$ for all $i\in [k]$ 
then $G$ contains an $\ell$-cycle for any 
$\ell  \in [2k+\ell _{\total}, 2(1-\eps )u |\mc{H}| + 
{\ell }_{\total} - 2k]$ of the same parity as $\ell_{\total}$.
\item If some $\{a_j,b_j\}$ with $j \in [k]$ is contained in a 
parity broken hub of $\mc{H}$ then $G$ contains an $\ell$-cycle for any 
$\ell  \in [7k+\ell _{\total}, 2(1-\eps )u |\mc{H}| + \ell _{\total} - 2k]$.
\end{enumerate}
\end{lem}

\begin{proof}
We write $\ell - {\ell}_{\total} = \sum _{i\in [k]} \ell'_i$,
where $\ell_j \ge 7$ (for (ii)), each $\ell'_i \ge 2$ with $i \ne j$
is a bipartite length for its hub, and for each $H \in \mc{H}$
we have $\sum \{ \ell'_i+1: \{a_i,b_i\} \sub V(H) \} \le 2(1-\eps)u$.
By the connection property of hubs,
and Lemma \ref{lem: parity broken hub lem} 
for the parity broken hub,
we can find vertex-disjoint $a_ib_i$-paths
of length $\ell'_i$ for each $i \in [k]$,
which combine with $\mc{P}$ to produce an $\ell$-cycle.
\end{proof}

\subsection{Breaking parity}

In this subsection we will prove that almost all hubs of $G$ are parity broken. 
This will use the bound on the independence number of $G$,
via the following proposition.

\begin{prop} \label{prop: indep disjoint matching}
Let $m,d,s \in {\mathbb N}$ with $m \geq 3d$. 
Suppose $G$ is a graph with $V(G) = \bigcup_{i\in [s]} I_i$, 
where $I_1,\ldots ,I_s$ are disjoint independent sets of order $m$. 
Suppose also that $\aA (G) < v(G)/ 12d$. 
Then there is $\{i_0,\ldots ,i_d\} \subset [s]$ 
and a matching of size $d$ with one edge in 
each $G[I_{i_{j-1}}, I_{i_j}]$ for $j\in [d]$.
\end{prop}

\begin{proof}
Consider a maximal matching ${\cal M}'$ in $G$ with the property 
that ${\cal M}'$ contains at most one edge of $G[I_i,I_j]$ 
for all distinct $i,j\in [s]$. We use ${\cal M}'$ 
to define a graph $H$ with $V(H) = [s]$, where
$ij \in E(H)$ if and only if ${\cal M}'$ contains an edge from $G[I_i,I_j]$. 
To prove the proposition, it suffices to show that $H$ contains
a path of length $d$. By Theorem \ref{thm: erdos-gallai},
it suffices to prove $d(H) > d-1$.\np
	
For contradiction, suppose $d(H) \leq d-1$. 
Let $S \sub V(H)$ with $|S|=s/2$ be such that
$d_H(i) \le d_H(j)$ for all $i \in S$, $j \notin S$.
Then $d_H(i) \le 2(d-1)$ for all $i \in S$.
By Theorem \ref{thm: Turan} (Tur\'an's Theorem),
there is an independent set $S' \subset S$ in $H$ 
with $|S'| \ge |S|/(2d-1) \geq s/4d$. 
For each $i \in S'$, let $J_i$ be obtained from $I_i$
by deleting all vertices contained in an edge of ${\cal M}'$. 
By the definition of ${\cal M}'$ and $S$,
we have $|J_i| \geq |I_i|-2d \geq m/3$. 
Since ${\cal M}'$ is maximal, there are no
edges between $J_i$ and $J_j$ for any distinct $i,j$,
so $\bigcup _{i\in S'} J_i$ is independent.
We deduce $\aA (G) \geq |S'|(m/3) \geq ms/12d = v(G)/12d$.
This contradiction completes the proof.
\end{proof}

We can now show that almost all $(u,\eps )$-hubs of $G$ are parity broken.

\begin{lem} \label{lem: parity breaking}
At least $(1-\eps )L$ hubs are parity broken.	 
\end{lem}

\begin{proof}
First we note that if $u \geq 4n$ then every hub $(A,B,D)$
must be parity broken. Indeed, as $\aA(G)<n$, any maximal 
matching in $A$ has size at least $u/3 > 2u^{1-\eps}$.
Thus we may assume $n \geq u/4 = \ell^{1- \eps }/4$.\np

For contradiction, suppose the hubs 
$\{(A_i,B_i,D_i)\}_{i\in [s]}$ are not parity broken, 
where $s = \eps L \geq \eps N/4u$. 
We will obtain a contradiction by using 
Lemma \ref{lem: hub lemma} to find an $\ell$-cycle.
Specifically, it suffices to show that there is a set of hubs
$\mc{H} = \{H_1,\dots,H_k\}$ for some $k \ge \ell/u$,
and a handle system $\mc{P}=\{P_1,\dots,P_k\}$ for $\mc{H}$,
where each $P_i$ has length $\ell_i$,
starts in $H_i$ and ends in $H_{i+1}$, and 
${\ell}_{\total} = \sum_{i\in [k]} \ell _i \le \ell/4$
has the same parity as $\ell$.\np

To achieve this, we look for a cycle of suitable length in the 
auxiliary graph $H$ with $V(H)=[s]$, where $ij \in E(H)$ if 
and only if there is an edge between $A_i$ and $A_j$. We apply 
Lemma \ref{cor: bfs decomposition} to $H$ with 
$\pow  = \ell ^{1- 2 \eps }$ to obtain
triples $\{(x_i,X_i,T_i)\}_{i\in [t]}$ so that 
for each $i\in [t]$ there is $d_i \in [0, \log_{\pow }(N)]$ 
such that each vertex of $X_i$ is at distance $d_i$
from $x_i$ in $T_i$. We let $X= \bigcup _{i\in [t]} X_i$ 
and note that $|X| \geq {s}/{2\pow } \geq \eps N / 8u\pow 
\ge \eps n\ell ^{-1 + 3\eps }/16$.\np 

We will construct a cycle by applying
Proposition \ref{prop: indep disjoint matching}
to find a long path in $H[X]$.
Consider a maximal matching in each $G[A_i]$ and let 
$I_i \subset A_i$ denote the vertices not covered by
the matching. Then each $I_i$ is independent and
$|I_i| \geq u/2$ as $(A_i,B_i,D_i)$ is not parity broken.
Deleting some vertices if necessary we can assume 
$|I_i| = u/2$ for all $i\in [s]$.
Fix $d \in \mb{N}$ of the same parity as $\ell$
with $d = \ell^{\eps } + 2 \pm 1$. Then $u/2 > 3d$,
and for large $\ell$ we have
$$ 
 \big|\bigcup _{i\in X} I_i\big|/12 d 
  \ge (\eps n\ell^{-1 + 3\eps }/16) \cdot (\ell^{1-\eps}/24d)
> n > \aA(G).
$$ 
Thus Proposition \ref{prop: indep disjoint matching}
applies to $G\big[\bigcup_{i\in X} I_i\big]$, giving
some $\{i_0,\ldots ,i_d\} \subset X$ 
and a matching $M$ of size $d$ with one edge in 
each $G[I_{i_{j-1}}, I_{i_j}]$ for $j\in [d]$. \np

Note that $P = i_0 \ldots i_d$ is a path in $H$,
so Lemma \ref{cor: bfs decomposition} (iii)
implies that it is contained in some $X_i$.
By the distance property of $X_i$, 
the unique $i_0 i_d$-path $Q$ in $T_i$
is internally disjoint from $P$
and has length $\ell(Q)$ which is even
with $\ell(Q) \le 2d_i \le 2\log_{\pow }(N)$.
Let $S$ be a set of edges obtained by choosing
one edge in $G[A_x,A_y]$ for each edge $xy$ of $Q$
(which exists by definition of $H$).
Then $M \cup S$ consists of a set of vertex-disjoint
paths, which we denote $P_1,\dots,P_k$,
with lengths $\ell_1,\dots,\ell_k$,
where $k \ge d-1 > \ell/u$ (as $M$ is a matching)
and ${\ell}_{\total} = \sum_{i\in [k]} \ell_i 
= \ell(Q) + d \le \ell/4$ has the same parity as $\ell$.
Furthermore, $\{P_1,\dots,P_k\}$ is a handle system
for a set of hubs $\{H_1,\dots,H_k\}$ such that 
each $P_i$ starts in $H_i$ and ends in $H_{i+1}$.
Now Lemma \ref{lem: hub lemma} (i) gives an $\ell$-cycle,
which is the required contradiction.
\end{proof}

\noi \emph{Remark:} 
Henceforth, we will assume all hubs of $G$ are parity broken. 
This can be guaranteed by taking $\eps$ slightly smaller 
in Corollary \ref{cor: decomp into hubs} and 
moving into $W$ any hubs that are not parity broken.

\subsection{Interaction between hubs}

We will now organise most of the hubs into `components',
so that there is no large matching between two hubs
in different components. To do so, we write
$U_i = A_i \cup B_i$ for each $i \in [L]$
and consider a maximum matching ${\cal M}$ 
in $G[ \cup _{i} U_i]$ such that 
(a) every $uv \in {\cal M}$ goes between distinct hubs, 
and (b) between any two distinct hubs there is at 
most one edge of ${\cal M}$.
We define an auxiliary graph $H_1$ on $[L]$
where $ij \in E(H_1)$ iff there is an edge 
of ${\cal M}$ between $U_i$ and $U_j$.
We start by bounding the average degree of $H_1$. \np

\begin{lem} \label{lem: average degree control of H}
$H_1$ has average degree at most $\ell^{1 - 3\eps }$.
\end{lem}

\begin{proof}
For contradiction, suppose $d(H_1) \geq \ell^{1-3\eps}$.
We apply Lemma \ref{lem: approx length cycle} to $H_1$ 
with $\pow = \ell^{1-5\eps}$ and $d_1 = \ell^\eps$,
noting that $d(H_1) \ge 16\pow d_1$,
to find an $\ell_1$-cycle for some 
$\ell_1\in [d_1,d_1 + 2 \log _{\gamma }(N)] 
\subset [d_1, d_1 + \ell/10]$,
using $\ell \geq C \log n / \log \log n$.
Its edges correspond to a submatching $\mc{M}'$
of $\mc{M}$ of size $\ell$, which forms a handle system
for a set of $\ell$ hubs. As $8\ell_1 \le \ell \le u\ell_1$ 
and each hub is parity broken, Lemma \ref{lem: hub lemma} (ii)
gives an $\ell$-cycle, which is a contradiction.
\end{proof}

By Lemma \ref{lem: average degree control of H},
at most $\eps L$ vertices of $H_1$ have degree greater 
than $\eps^{-1}\ell ^{1-3\eps }$ in $H_1$. 
Let $H_2$ be obtained from $H_1$ by deleting these 
high degree vertices, so that $v(H_2) \ge (1-\eps)L$.
We will now restrict attention to the subgraph $H_3$ of $H_2$
where $ij \in E(H_3)$ iff $G[U_i,U_j]$ has a matching
of size $2\ell^{\eps }$. We show that $H_3$ 
does not have large components.

\begin{lem} \label{lem: connected component of H_2}
All connected components of $H_3$ have
fewer than $(1+2 \eps )\ell^{\eps }/2$ vertices.
\end{lem}

\begin{proof}
For contradiction, suppose $H_3$ contains a tree $T$ 
with $(1+2\eps )\ell^{\eps }/2$ vertices. 
By definition of $H_3$, we can greedily choose 
a matching $\mc{P} = \{ P_1,\dots,P_k\}$ that
contains two edges of $G[U_i,U_j]$ for each $ij \in E(T)$.
We can regard $\mc{P}$ as a handle system for the
hubs $\mc{H} = \{ (A_i,B_i,D_i): i \in V(T) \}$.
To see this, we note that condition (i) is vacuous,
and (iii) holds as $|\mc{P}|=2e(T) < \ell^\eps < u^{1-\eps}/2$.
To achieve (ii), we order the edges of $\mc{P}$ 
cyclically according to a closed walk in $T$ that uses 
every edge exactly twice (which is well-known to exist,
e.g.\ by embedding $T$ in the plane 
and walking around its outside). 
As $8|\mc{P}| \le \ell \le 2(1-\eps )u |\mc{P}| - |\mc{P}|$
and all hubs are parity broken, Lemma \ref{lem: hub lemma} (ii)
gives an $\ell$-cycle, which is a contradiction.
\end{proof}

\subsection{Proof of stability}

We now combine the results of this section to prove our stability result.

\begin{proof}[Proof of Lemma \ref{lem: almost clique decomp}]
Let the graphs $H_1$, $H_2$ and $H_3$ be as in the previous subsection.
Fix a maximal matching $M_{ij}$ in $G[U_i,U_j]$
for each $ij \in E(H_2) \sm E(H_3)$;
by definition of $H_3$ each $|M_{ij}| \le 2\ell^\eps$.
For each $i \in V(H_2)=V(H_3)$, 
let $U'_i = U_i \sm \bigcup_{ij} V(M_{ij})$;
by definition of $H_2$ each 
$|U'_i| \ge |U_i| - \eps^{-1}\ell ^{1-3\eps} \cdot 2\ell^\eps
\ge (1-\eps)2u$ for large $\ell$.
Let $U' = \bigcup \{ U'_i: i \in V(H_3) \}$ and $G'=G[U']$.
We have $|U'| \ge |V(H_3)| \cdot (1-\eps)2u
\ge (1-\eps)^2 2uL \ge (1-3\eps)N$,
so by Theorem \ref{thm: Turan} (Tur\'an's Theorem)
$d(G') \ge (1-4\eps )\ell$.\np

Note that all edges of $G'$ lie within some hub
or join two hubs in the same connected component of $H_3$.
By Lemma \ref{lem: connected component of H_2}
the number of vertices in any component of $G'$ is at most
$(1+2\eps )(\ell^{\eps }/2) \cdot (2 \ell ^{1-\eps }) 
= (1 + 2\eps )\ell$. Let $B$ be obtained from $U'$
by deleting $V(C)$ for any component $C$ of $G'$
with $d(C) \le (1-\eps^{1/2})\ell$. Then
$ |U'|(1-4\eps)\ell \le 2e(G') 
\le |B| (1+\eps)\ell + (|U'|-|B|) (1-\eps^{1/2})\ell$,
which gives 
$|B|(\eps^{1/2} + \eps ) \geq |U'|(\eps^{1/2} - 4\eps )\ell$,
and so $|B| \geq (1-6\eps^{1/2})|U'| \geq (1-7\eps^{1/2})N$.\np 

We conclude by taking subgraphs of high minimum degree
in each component of $G'[B]$. 
Letting $k = (1-\eta/2)\ell$, each such component $C$ has 
$e(C) = d(C)v(C)/2 \geq (1-\eps ^{1/2})\ell v(C)/2 
\geq \binom {k}{2} + (v(C)-k)(1-\eta /2)\ell$,
as  $(1+\eps )\ell \geq v(C) \geq d(C) \geq (1-\eps ^{1/2})\ell $ 
and $\eps \ll \eta $. Proposition \ref{prop: moving to the k-core} (i) 
gives a subgraph $C'$ of $C$ with
$\dD(C') \ge k \geq (1-\eta /2)\ell  \geq (1-3\eta /4)v(C)$.
We let $V_1,\ldots , V_s$ be the vertex-sets 
of these subgraphs $C'$ for all components $C$ of $G'[B]$. 
Then each $|V_i| \ge \dD(G[V_i]) \ge (1-\eta)\ell$ and
$\sum_{i=1}^s |V_i| \ge (1-3\eta /4)|B| \geq (1-\eta )N$.
Lastly, suppose for contradiction that some $|V_i| \geq \ell$. 
We may delete $|V_i| - \ell \leq 3\eps \ell $ vertices from $V_i$ 
and apply Theorem \ref{thm: dirac} (Dirac's Theorem)
to find an $\ell $-cycle in $G$. This contradiction
shows that all $|V_i| \leq \ell -1$.
\end{proof}

\section{The upper bound} \label{sec:pf}

In this section we will prove our main result,
Theorem \ref{thm: clique-vs-cycle}, which establishes
the upper bound on cycle-complete Ramsey numbers;
the proof will be given in the last subsection.
Most of this section will be occupied with
cleaning up the approximate structure of a supposed
counterexample, as provided by the stability result 
in the last section, until it becomes clear that 
its properties are contradictory, so it cannot exist. \np

Throughout the section we fix a graph $G$ and 
`approximate cliques' $V_1,\ldots, V_s$ satisfying 
the hypotheses and conclusions of 
Lemma \ref{lem: almost clique decomp}.
In the first subsection we give conditions
under which the approximate cliques can absorb
additional vertices from the remainder 
$R := V(G) \sm \bigcup_{i=1}^s V_i$,
while maintaining pancyclicity and also the property
that any pair of vertices can be connected by paths
with a large range of possible lengths.
In the second subsection we clean up $R$
by absorbing some of its vertices into
the approximate cliques.
In the third subsection we show that the
remaining part of $R$ can be separated from
most of the approximate cliques,
in the sense they have each have a large 
subset with no neighbours in $R$.
In the fourth subsection we show that one
of the approximate cliques has a vertex
that can absorb its neighbours.
This final property quickly leads to 
a contradiction, which will complete the proof.

\subsection{Absorbable paths}

In this subsection we consider the 
following set-up which is very similar
to the handle systems used for hubs.
Given a set of paths ${\cal P} = \{P_1,\ldots ,P_m\}$ 
in a graph $H$ and a set $V \sub V(H)$,
we say $\mc{P}$ is absorbable into $V$
if it consists of paths that are
vertex-disjoint and disjoint from $V$,
and there are distinct vertices 
$\{a_1,\ldots ,a_m, b_1,\ldots b_m\} \subset V$ 
such that $a_i$ is adjacent to one end of $P_i$
and $b_i$ is adjacent to the other end of $P_i$;
we say that $P_i$ attaches to $a_i$ and $b_i$.
The following lemma will be used to 
absorb paths into approximate cliques.\np
	
\begin{lem} \label{lem: cycle of specified length}
Let $H$ be a graph with a partition $V(H) = U \cup V$, 
where $\dD(H[V]) \geq 0.9|V|$ and $|U| \leq 0.1|V|$. 
Suppose that ${\cal P}$ is a set of paths of
length at most $2$ which is absorbable into $V$
and has $\cup_{P \in {\cal P}}V(P) = U$. Then 
\begin{enumerate}[(i)]
\item $H$ contains an $xy$-path of length $\ell$
for any distinct $x,y$ in $V(H)$
and $\ell  \in [6, 2v(H)/3]$,
\item $H$ is pancyclic.
\end{enumerate} 
\end{lem}
	
\begin{proof}
For (i), we suppose first that both $x$ and $y$ are in $V$,
and show that there is an $xy$-path of length $\ell$
for any $\ell \in [2, 2v(H)/3]$. To see this, we use
$\dD(H[V]) \geq 0.9|V|$ to greedily choose an $xy'$-path $P$
of length $\ell-2$ in $H[V]$ that avoids $y$.
As $|N_{H[V]}(y) \cap N_{H[V]}(y')| - |V(P)|
\ge 0.8|V| - (2/3)1.1|V| > 0$ we can choose
a common neighbour of $y$ and $y'$ in $V \sm V(P)$,
and so obtain the required $xy$-path of length $\ell$.
Next we suppose that $x$ is in $V$ and $y$ is in $U$.
Then $y$ lies on a path $P \in \mc{P}$. Let $a$ and $b$
be the attachments of $P$, where without loss of generality 
$a \ne x$. The subpath of $P$ from $y$ to $a$ has
length $\ell' \le 3$. Adding a path of length $\ell-\ell'$
from $a$ to $x$ gives the required $xy$-path of length $\ell$.
Finally, suppose $x$ and $y$ are both in $U$.
Then we can find $a$ and $b$ in $V$ so that there is
an $xa$-path and $yb$-path that are vertex-disjoint
and both of length at most $2$. Adding an $ab$-path
of the appropriate length completes the proof of (i). \np

For (ii), we first note that by Theorem \ref{thm: bondy}
(Bondy's Theorem) $H[V]$ is pancyclic. It remains to show
there is an $\ell$-cycle whenever $|V| < \ell \le |V(H)|$.
Let $S$ be the set of attachments of $\mc{P}$, and
fix any $V' \sub V \sm S$ with $|V'| = \ell - |S| - |U|$.
As $|U|+|S| \le 3|U| \le 0.3|V|$, we have $|V'| \ge 0.7|V|$.
Let $H'$ be the graph obtained from $H[V']$
by adding a new vertex $v_P$ for each $P \in \mc{P}$,
which is joined to all common neighbours
in $V'$ of the attachments of $P$.
Note that $v(H')=|V'|+|\mc{P}|$ and 
$\dD(H') \ge |V'| - 0.2|V| \ge |V|/2 > v(H')/2$,
and so by Theorem \ref{thm: dirac} (Dirac's Theorem)
$H'$ has a Hamilton cycle. Replacing each $v_P$
by $P$ and the edges to its attachments
produces a cycle of length $\ell$ in $H$, as required.
\end{proof}

\subsection{Cleaning up the remainder}	

Here we clean up the remainder 
$R = V(G) \sm \bigcup_{i=1}^s V_i$
by absorbing some of its vertices into the approximate cliques, 
according to the following algorithm.
For each $i\in [s]$ we keep track of two sets during the algorithm: 
(a) a set $W_i = V_i \cup R_i$, where $R_i \sub R$ 
has been absorbed by $V_i$, and
(b) a subset $A_i$ of $V_i$, which is available 
for further attachments in the sense of the previous subsection.
We start with $W_i = A_i = V_i$ for each $i\in [s]$. In a given round: 
\begin{itemize}
\item Consider any path $P$ of length at most $2$ in $G[R]$
that attaches to some  distinct vertices $a$, $b$ in $A_i$ 
for some $i \in [s]$. If there is no such $P$ then stop.
Otherwise, move $V(P)$ from $R$ to $R_i$,
delete $a$ and $b$ from $A_i$, and proceed to the next round.
\end{itemize} 
We claim that the algorithm terminates
with $|W_i| \leq \ell -1$ for all $i\in [s]$.
Indeed, otherwise in some round some $|W_i| \in [\ell,\ell+2]$, 
as $W_i$ increments by at most $3$ vertices in each round. 
Then $W_i$ has a partition $W_i = V_i \cup R_i$, 
where $|V_i| \geq \dD(G[V_i]) \geq (1-\eta )|V_i| \ge 0.9|V_i|$ 
and $|R_i| \leq \eta \ell +2 \leq 0.1|V_i|$. 
By construction, $R_i$ is the union of paths $\mc{P}_i$ absorbable
into $V_i$, so Lemma \ref{lem: cycle of specified length} (ii)
gives an $\ell$-cycle in $G[W_i]$. 
This contradiction proves the claim.
We deduce $|R_i|=|W_i|-|V_i| < \eta \ell$.
Furthermore, each $A_i$ decreased by two vertices for each path 
added to $R_i$, so $|A_i| \ge |V_i| - 2|R_i| \ge (1-3\eta )\ell$.\np 

\subsection{Separating the remainder}

Now we show that the cleaned up remainder
$R := V(G) \sm \bigcup_{i=1}^s W_i$ can be separated 
from most of the approximate cliques, in the following sense.
For $i\in [s]$ let $A_i'$ be the set of $v \in A_i$
such that $v$ has a neighbour in $R$.
We partition $[s]$ as $S \cup T$,
where $T = \{i \in [s]: |A_i'| < \ell^{2/3}\}$.

\begin{lem} \label{T} $|T| \geq s/2$. \end{lem}

\begin{proof}
We start by constructing a partition $R = U_1 \cup \cdots \cup U_r$, 
where each $G[U_j]$ has diameter at most $2$ and $r\leq 2N\ell ^{-1/2}$. 
To see that this is possible, we repeatedly remove stars from $R$ 
of order $\ell^{1/2}$ until none remain. 
We can remove at most $\eta N/\ell^{1/2}$ such stars.
The remaining set $R'$ must have $d(G[R']) < \ell ^{1/2}-1$. 
By Theorem \ref{thm: Turan} (Tur\'an's Theorem) 
$|R'|/\ell^{1/2} \leq \aA(G[R']) \le \aA (G) < n$, so 
$|R'| < n \ell^{1/2} < \tfrac{3}{2} N \ell ^{-1/2}$. 
We let the parts $U_1,\dots,U_r$ consist of all removed stars 
and singleton parts for each vertex of $R'$.
Then $r\leq 2N\ell ^{-1/2}$, as required.\np 

Now suppose for contradiction that $|T|<s/2$, so $|S|>s/2$.
For each $v \in \bigcup_{i=1}^s A'_i$ we fix any $u_v \in N(v) \cap R$.
We consider an auxiliary bipartite graph $H$ with parts 
$A = \{W_i\}_{i\in [s]}$ and $B = \{U_j\}_{j\in [r]}$,
where we add an edge from $W_i$ to $U_j$
for each $v \in A'_i$ with $u_v \in U_j$.
To see that this gives a (simple) graph we use the termination
condition of the algorithm in the previous subsection:
there cannot be distinct $v_1,v_2 \in A'_i$ with neighbours $u_1,u_2 \in U_j$, 
as $U_j$ has diameter at most $2$, so we would have a $u_1u_2$-path 
of length at most $2$ attaching to $A_i$.\np

We will obtain a contradiction by finding a short cycle in $H$
and using it to construct an $\ell$-cycle in $G$.
We have $v(H) = s + r \leq 2N \ell ^{-1} + 2N \ell ^{-1/2} \leq 4N\ell ^{-1/2}$
and $e(H) = \sum_{i\in [s]} |A_i'| \ge |S|\ell ^{2/3}
> \tfrac{1}{2} \tfrac{N}{2\ell} \ell ^{2/3}
\ge \ell ^{1/6}v(H)/16$, so $d(H) > \ell^{1/6}/8$.
As $\ell \ge C\frac {\log n}{\log \log n}$,
we can apply Lemma \ref{lem: approx length cycle}
with $\pow = d_1 = \ell ^{1/14}$ to find
a cycle in $H$ with length in
$[\ell ^{1/14}, \ell ^{1/14} + 2 \log _{\ell ^{1/14}}(v(H))] 
\subset [4, \ell/8]$. \np

As $H$ is bipartite, we can write this cycle as
$W_{i_1}U_{i_1}W_{i_2}\cdots W_{i_L}U_{i_L}W_{i_1}$,
for some $2 \le L \le \ell/16$.
Each $U_{i_j}$ has diameter at most $2$,
so by construction of $H$ there is a path $Q_j$
of length at most~$4$, starting with the edge
$b_j u_{b_j}$ for some $b_j \in W_{i_j}$
and ending with the edge $u_{a_{j+1}} a_{j+1}$
for some $a_{j+1} \in W_{i_{j+1}}$.
Furthermore, $a_j, b_j \in W_{i_j}$ are distinct,
as $u_{a_j} \ne u_{b_j}$. We fix 
$\ell_j \in [2, \ell/2]$ for each $j \in [L]$ with
$\sum_{j\in [L]} \ell _j = \ell - \sum_{j\in [L]} e(Q_j)$.
and apply Lemma \ref{lem: cycle of specified length} (i)
to choose $a_jb_j$-paths $P_j$ in $W_j$ of length $\ell_j$. 
Combining these with the paths $Q_j$
produces an $\ell$-cycle, which is a contradiction.
\end{proof}

\subsection{Absorbing neighbours}

Now we will show that one of the approximate cliques 
has a vertex that can absorb its neighbours.
To do so, we now analyse the edges
crossing between the approximate cliques.
For each $i \in T$ let $B_i = A_i \sm A'_i$
denote the set of $v \in A_i$ with no neighbour in $R$.
By definition of $T$ each 
$|B_i| \geq |A_i| - \ell ^{2/3} \geq 2|V_i|/3$. 
For each $i\in T$ we consider a matching ${\cal M}_i$ in $G$ 
of maximum size subject to the condition that each edge 
of ${\cal M}_i$ intersects $W_i$ in a single vertex from $B_i$.
We will show that these matchings cannot all be large.

\begin{lem} \label{smallMi}
There is $i^* \in T$ with
$|{\cal M}_{i^*}| \leq \ell ^{1/3}$.
\end{lem}

Before giving the proof, we show how this lemma allows us 
to find a vertex that can absorb its neighbours.
Recall that $W_{i^*} = V_{i^*} \cup R_{i^*}$ and 
$R_{i^*}$ is a union of vertex-disjoint paths $\mc{P}_{i^*}$
that is absorbable into $V_{i^*}$.

\begin{lem} \label{absorbN}
There is $v \in B_{i^*}$ such that for any
neighbours $y_1,\dots,y_k$ of $v$ in $\ov{W_{i^*}}$
with $k \le \ell - |W_{i^*}|$, letting $\mc{P}'_{i^*}$
be obtained from $\mc{P}_{i^*}$ by adding each $y_i$
as a path of length $0$, we have $\mc{P}'_{i^*}$
absorbable into $V_{i^*}$. 
\end{lem}

\begin{proof}
We apply the following algorithm to construct a set $D_{i^*} \sub B_{i^*}$ 
such that every vertex in $B_{i^*}$ has the stated property.
We start with $D_{i^*}=B_{i^*}$ and $X=\es$.
While there is $x \in \bigcup _{j\neq i^*} W_j$ 
with $1 \leq d_G(x,D_{i^*}) \leq 2 \ell ^{1/3}$
we add $x$ to $X$ and delete $N_G(x) \cap D_{i^*}$ from $D_{i^*}$.
This process terminates with a set $D_{i^*}$ such that  
$d_G(x,D_{i^*}) = 0$ or $d_G(x,D_{i^*}) > 2 \ell ^{1/3}$ 
for all $x \in ( \bigcup _{j\neq i}W_j ) \sm X$.
Each $x \in X$ has a private neighbour in $B_{i^*}$,
so by choice of $\mc{M}_{i^*}$ we have
$|X| \leq |{\cal M}_{i^*}| \leq \ell ^{1/3}$, and so 
$|D_{i^*}| \geq |B_{i^*}| - (2\ell ^{1/3})|X| \geq \ell /2 > 0$.\np

Consider any $v \in D_{i^*}$ and neighbours $y_1,\dots,y_k$ 
of $v$ in $\ov{W_{i^*}}$ with $k \le \ell - |W_{i^*}|$.
Each $y_i$ is not in $X$ (otherwise we would have
deleted $v$ from $D_{i^*}$) so has at least $2 \ell ^{1/3}$ 
neighbours in $D_{i^*}$. This implies 
$k \le |\mc{M}_{i^*}| \le \ell ^{1/3}$,
or otherwise we could greedily construct
a matching of size $|\mc{M}_{i^*}|+1$ between
$\{y_1,\dots,y_k\}$ and $B_{i^*}$, which is
contrary to the choice of $\mc{M}_{i^*}$. 
We can therefore greedily choose two
attachments for each $y_i$ in $D_{i^*}$,
which are distinct from each other,
and distinct from the attachments of $\mc{P}_{i^*}$
as $D_{i^*} \sub B_{i^*} \sub A_{i^*}$. 
Thus $\mc{P}'_{i^*}$ is absorbable into $V_{i^*}$. 
\end{proof}

We conclude this subsection by returning 
to the proof of Lemma \ref{smallMi}.

\begin{proof}[Proof of Lemma \ref{smallMi}]
For contradiction, suppose $|{\cal M}_i|>\ell ^{1/3}$ for all $i \in T$.
Note that every edge in ${\cal M}_i$ has one end in $B_i$
and the other end in $\cup_{j \ne i} W_j$
(it is not in $R$ by definition of $B_i$).
Consider a uniformly random partition $[s] = S_1 \cup S_2$.
Say that $bc \in {\cal M}_i$ with $b \in B_i$ and $c\in W_j$ 
is \emph{good} if $i\in S_1$ and $j\in S_2$. 
Each edge is good with probability $1/4$, so we can fix a partition 
so that the number of good edges is at least
$\tfrac{1}{4}\sum_{i \in T} |\mc{M}_i|
> |T| \ell ^{1/3} /4 \ge s \ell ^{1/3}/8$. \np

Consider the auxiliary bipartite graph $H$ with parts 
$A = \{W_i\}_{i\in S_1}$ and $B = \{W_j\}_{j\in S_2}$,
where we add an edge from $W_i \in A$ to $W_j \in B$ for each 
good edge $bc \in {\cal M}_i$ with $b \in B_i$ and $c\in W_j$.
We claim that $H$ is a (simple) graph.
To see this, suppose on the contrary we have
$b_{1}c_{1}$ and $b_{2}c_{2}$ in ${\cal M}_i$
with $\{c_1,c_2\} \sub W_j$.
By Lemma \ref{lem: cycle of specified length} (i) there is 
a $b_{1}b_{2}$-path $P_1$ in $G[W_i]$ of length $\bfl{\ell/2}-1$
and a $c_{1}c_{2}$-path $P_2$ in $G[W_j]$ of length $\bcl{\ell/2}-1$.
Combining the paths $P_1$ and $P_2$ with the edges $b_{1}c_{1}$ 
and $b_{2}c_{2}$ gives a $\ell $-cycle.
This contradiction proves the claim. \np

We deduce $e(H) \ge s \ell ^{1/3}/8 = v(H) \ell ^{1/3}/8$,
so $d(H) \ge \ell ^{1/3}/4$. We use this to obtain the required
contradiction by finding a short cycle in $H$, and so an $\ell$-cycle 
in $G$. This part of the proof is very similar to that of Lemma \ref{T}.
Lemma \ref{lem: approx length cycle} provides an even cycle
$W_{i_1}W_{j_1}\cdots W_{i_L}W_{j_L}W_{i_1}$,
for some $2 \le L \le \ell/16$,
where each $i_\aA \in S_1$ and $j_\aA \in S_2$.
By definition of $H$, for each $\aA \in [L]$
there are edges $a_\aA x_\aA$ and $b_\aA y_\aA$ 
in ${\cal M}_{i_\aA}$ with $\{a_\aA,b_\aA\} \sub W_{i_\aA}$,
$x_\aA \in W_{j_\aA-1}$ and $y_\aA \in W_{j_\aA}$.  
By Lemma \ref{lem: cycle of specified length} (i)
there is a path $Q_\aA$ of length at most $4$ 
from $b_{\aA-1}$ to $a_\aA$ through $W_{j_\aA-1}$
via $y_{\aA-1}$ and $x_\aA$ (whether or not these coincide).
We fix $\ell_\aA \in [2, \ell/2]$ for each $\aA \in [L]$ with
$\sum_{\aA \in [L]} \ell_\aA = \ell - \sum_{\aA\in [L]} e(Q_\aA)$.
and apply Lemma \ref{lem: cycle of specified length} (i)
to choose $a_\aA b_\aA$-paths $P_\aA$ in $W_{i_\aA}$ 
of length $\ell_\aA$. Combining these with the paths $Q_\aA$
produces an $\ell$-cycle, and so the required contradiction.
\end{proof}

\subsection{Proof of Theorem \ref{thm: clique-vs-cycle}}	
	
We now complete the proof of our main theorem.
	
\begin{proof}[Proof of Theorem \ref{thm: clique-vs-cycle}]
We fix $\ell \in \mb{N}$ and prove the following statement (*)
by induction on $n \ge 1$ such that if $n \ge 3$
we have $\ell \geq C\frac {\log n}{\log \log n}$
(for some large absolute constant $C$):\np

(*) there is no $C_\ell$-free graph $G$ 
with $v(G) = N = (\ell -1)(n-1)+1$ and $\aA(G) \leq n-1$.\np

The case $n=1$ holds as every graph $G$ with $v(G) \geq 1$ 
has an independent set of order $1$.
The case $n=2$ holds as every graph $G$ with $v(G)\geq \ell$ 
contains an independent set of order $2$ or a clique of order $\ell$.

Now we give the induction step for $n \ge 3$.
For contradiction, suppose we have a 
$C_\ell$-free graph~$G$ 
with $v(G) = N = (\ell -1)(n-1)+1$ and $\aA(G) \leq n-1$.\np

If there is any vertex $v$ of degree less than $\ell-1$
we delete $N(v) \cup \{v\}$ from $G$ and apply induction.
The remaining subgraph $G_1$ satisfies 
$v(G_1) \geq (\ell-1)(n-2)+1 = r(C_{\ell },K_{n-1})$ by induction, 
so it contains a cycle $C$ of length $\ell$ or an independent set $I$ 
of order $n-1$. Then $G$ contains an $\ell$-cycle $C$
or $\{v\} \cup I$ forms an independent set of order $n$. 
Thus we may assume $\dD(G) \geq \ell-1$.\np

We let $V_1,\dots,V_s$ be the approximate cliques provided
by the stability result (Lemma \ref{lem: almost clique decomp}),
let $W_i = V_i \cup R_i$ for $i \in [s]$ be the enlarged
approximate cliques obtained in the previous section
by absorbing part of the remainder,
and let $v \in B_{i^*}$ be given by Lemma \ref{absorbN}.
As $v$ has at least $\ell-1$ neighbours, we can choose
neighbours $y_1,\dots,y_k$ of $v$ in $\ov{W_{i^*}}$
with $k = \ell - |W_{i^*}|$. As the path system $\mc{P}'_{i^*}$
in Lemma \ref{absorbN} is absorbable,
Lemma \ref{lem: cycle of specified length} gives a cycle
of length $|V_i|+|R_i|+k = \ell$ in $G$.
This gives a contradiction and completes the proof of the theorem.
\end{proof}	

\section{Concluding remarks}

Our results answer the questions of Erd\H{o}s et al. \cite{EFRS} up to a constant 
factor, which we did not compute explicitly, although with more work it 
seems that a reasonable value (less than $20$, say) can be obtained. It 
would be interesting to obtain an asymptotic formula for the $\ell$ 
minimising $r(C_{\ell }, K_n)$.
The constructions for the lower bound  on $r(C_{\ell }, K_n)$ avoid a 
range of cycles. For large $\ell$, this range consists of all cycles of 
length at least $\ell$, and for small $\ell$, it consists of all cycles 
of length at most $\ell$. This suggests that the finer nature of the 
threshold may be connected to the problem of improving the Moore bound
(see \cite{MS}) on the number of edges in a graph of given order and 
diameter.\np

The problem of obtaining good estimates on $r(C_{\ell }, K_n)$ 
for small $\ell>3$ remains widely open. The most significant gap
in the current state of knowledge is the case $\ell=4$,
for which the known bounds (see \cite{C-L-R-Z, Spen-CCRN}) 
are $c (n/\log n)^{3/2} \le r(C_4,K_n) \le C(n/\log n)^2$
for some constants $c$ and $C$.\np

\noi \textbf{Acknowledgement.}
The authors would like to thank the organisers of 
the Workshop on Extremal and Structural Combinatorics 
held at IMPA in Rio de Janeiro, where this work began.

\end{document}